\documentclass[5p,times]{elsarticle}
\usepackage{amsmath,amssymb,amsthm}
\usepackage{algorithm}
\usepackage{algorithmic}
\usepackage{graphicx}
\usepackage{booktabs}
\usepackage[section]{placeins}
\usepackage{multirow}
\usepackage[x11names,table]{xcolor}

\counterwithout{equation}{section}
\usepackage{makecell}
\usepackage{lineno}

\usepackage[colorlinks=true,           
            linkcolor=blue,            
            citecolor=blue,            
            urlcolor=blue]{hyperref}

\theoremstyle{plain}

\theoremstyle{remark}

\newtheorem{proposition}{Proposition}[section]

\DeclareMathOperator*{\minimize}{minimize}

\DeclareMathOperator*{\subt}{subject \ to}
\begin{document}
\begin{frontmatter}

\title{
Learning to control inexact Benders decomposition via reinforcement learning}
\author[inst1]{Zhe Li}

\author[inst1]{Bernard T. Agyeman}

\author[inst2]{Ilias Mitrai\corref{cor1}}
\ead{imitrai@che.utexas.edu}

\author[inst1]{Prodromos Daoutidis\corref{cor1}}
\ead{daout001@umn.edu}

\cortext[cor1]{Corresponding author}

\affiliation[inst1]{%
  organization={Department of Chemical Engineering and Material Science, University of Minnesota},
  city={Minneapolis},
  postcode={MN 55455},
  country={United States}
}

\affiliation[inst2]{%
  organization={McKetta Department of Chemical Engineering, University of Texas at Austin},
  city={Austin},
  postcode={TX 78712},
  country={United States}
}

\begin{abstract}
    Benders decomposition (BD), along with its generalized version (GBD), is a widely used algorithm for solving large-scale mixed-integer optimization problems that arise in the operation of process systems. However, the off-the-shelf application to online settings can be computationally inefficient due to the repeated solution of the master problem. An approach to reduce the solution time is to solve the master problem to local optimality. However, identifying the level of suboptimality at each iteration that minimizes the total solution time is nontrivial. In this paper, we propose the application of reinforcement learning to determine the best optimality gap at each GBD iteration. First, we show that the inexact GBD can converge to the optimal solution given a properly designed optimality gap schedule. Next, leveraging reinforcement learning, we learn a policy that minimizes the total solution time, balancing the solution time per iteration with optimality gap improvement. In the resulting RL-iGBD algorithm, the policy adapts the optimality gap at each iteration based on the features of the problem and the solution progress. In numerical experiments on a mixed-integer economic model predictive control problem, we show that the proposed RL-enhanced iGBD method achieves substantial reductions in solution time.

\end{abstract}


\end{frontmatter}

\section{Introduction}\label{sec:intro}
Benders decomposition (BD) \cite{benders1962partitioning}, along with its generalized version (GBD) \cite{gbd}, is a cutting-plane-based hierarchical decomposition algorithm that has been widely used in process systems engineering (PSE) to solve large-scale optimization problems. Typical applications include production planning and scheduling \cite{elcci2022stochastic}, optimization under uncertainty \cite{laporte1993integer, you2013multicut, luo2024design}, capacity expansion \cite{lohmann2017tailored, lara2018deterministic}, integrated scheduling and control \cite{chu2013integrated, mitrai2022multicut}, mixed integer optimal control \cite{menta2020learning, bansal2003new}, and machine learning \cite{aytug2015feature, santana2025support}. 

BD exploits the underlying structure of an optimization problem, specifically, the presence of a set of variables called complicating variables, which, if fixed, renders the resulting problem easier to solve. In the case of mixed-integer optimization problems, the original problem is decomposed into two smaller problems: a master problem, which considers the integer and possibly some continuous variables, and a subproblem, which is a continuous optimization problem. The master problem and subproblem are solved sequentially and coordinated via Benders cuts, which inform the master problem about the effect of the complicating variables on the subproblems. 

Although BD can lead to a reduction in solution time compared to a monolithic solution approach, its off-the-shelf application can be inefficient, especially for online applications. The computational efficiency of BD depends on the computational effort required to solve the master problem and subproblem at each iteration, and the quality of the Benders cuts, which determine the coordination efficiency. Over the years, several approaches have been proposed to accelerate Benders-type decomposition algorithms by reducing the number of iterations or the solution time per iteration \cite{rahmaniani_bd-review_2017}. Initial efforts exploited the mathematical properties of the monolithic, master, and subproblems. For example, valid inequalities \cite{saharidis2011initialization} can be used to reduce the number of iterations where the subproblem is infeasible. Another approach is to implement the multicut version of the algorithm by decomposing the original problem or a reformulated version into multiple subproblems \cite{birge1988multicut, laporte1993integer,you2013multicut}. Other acceleration techniques focus on improving coordination by generating multiple types of cuts \cite{magnanti1981accelerating, saharidis2010improving, fischetti2010note, sherali2013generating, bodur2017strengthened, glushko2022shaped}. These techniques aim to maximize the amount of information exchanged between the master problem and the subproblem, leading to a reduction in the number of iterations required for convergence.

\begin{figure*}[ht!]
    \centering
    \includegraphics[scale=0.85]{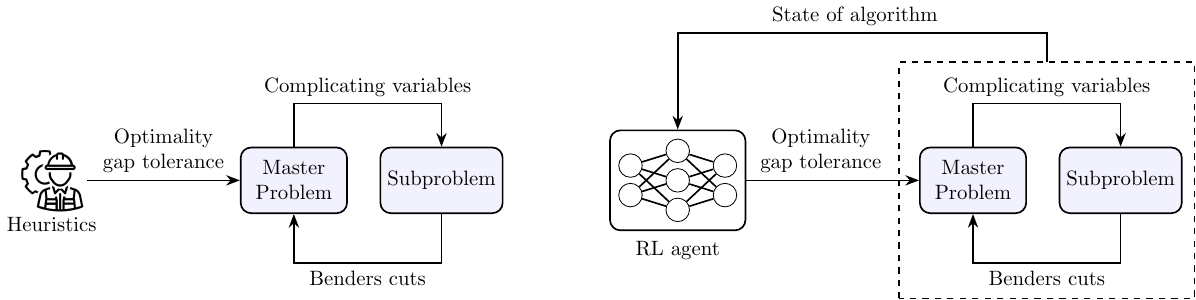}
    \caption{Comparison of traditional iGBD and the RL-enhanced iGBD framework.}
    \label{fig:RL approach}
\end{figure*}

An alternative is to reduce the solution time per iteration by solving the master problem and/or subproblem inexactly, i.e., obtaining a feasible solution that is not necessarily locally or globally optimal. This is a common approach in hierarchical (cutting-plane) \cite{holmberg1994using, rei2009accelerating,contreras2011benders,  li2013inexact, tsang_inexact_2022} and distributed decomposition-based algorithms \cite{boyd2011distributed, xie2017inexact}. In the context of BD, it has been used to accelerate the solution of the master problem and the subproblem, leading to the so-called inexact BD algorithm. This approach aims to exploit the fact that in early iterations the master problem does not have enough Benders cuts and the values of the complicating variables are far from optimal. Thus, heuristics can be used to find locally optimal integer feasible solutions for the master problem or to terminate the solution of the subproblem early \cite{easwaran2009tabu, poojari_genetic_2009, jiang_tabu-bd_2009, lai_gene-bd-capa_2010, Lai201233, geoffrion_multicommodity, rei2009accelerating}. The main challenge with such inexact solution approaches, especially when applied to the master problem, is guaranteeing finite convergence. Traditionally, this was achieved by properly designing an optimality gap schedule for the master problem (which is a mixed integer problem), which reduces the optimality gap monotonically as iterations progress. Although this strategy leads to convergence, it is not necessarily optimal with respect to the total solution time. The task of finding the best parameters for a solution algorithm that minimizes a desired performance metric, such as the solution time, is a black-box optimization problem known as the algorithm configuration task \cite{adriaensen2022autodac}. Existing algorithm configuration approaches rely on machine learning (ML) to learn from data the relation between the configuration of a solution algorithm and its solution time. In the case of BD, ML techniques have been used to identify valuable cuts to add at each iteration \cite{jia_benders_2021, allen2023improvements}, select a set of cuts to add as a warm start \cite{mitrai_gbd_init_2024, mitrai2024machine}, predict integer complicating variables \cite{mana_accelerating_nodate}, and approximate Benders cuts \cite{mitrai_computationally_2024}.

We propose the application of reinforcement learning (RL) to determine the best optimality gap tolerance at each iteration, so that the solution time of the inexact BD algorithm is minimized. The proposed approach is motivated by the sequential nature of the algorithm, since the optimality gap tolerance must be determined at each iteration based on solution progress. RL has been used to solve sequential decision-making problems arising in process control \cite{shin2019reinforcement, nian2020review, yoo2021reinforcement}, scheduling \cite{hubbs2020deep}, supply chain management \cite{shin2016multi, wang2025risk, burtea2024constrained, kotecha2025leveraging}, capacity expansion \cite{shin2019multi}, and process design \cite{gao2024deep, khan2020searching}. RL has also been used to accelerate optimization algorithms \cite{mitrai_review-opt-ml_2024, bengio2021machine} for mixed integer optimization \cite{tang_reinforcement_2020, wang_cut-hem_2023, chi_rl-cg_2023, yuan_rl-cg_2024}. 

In this paper, we first develop the inexact solution approach for GBD, which is referred to as inexact GBD (iGBD) in the sequel. We show that the algorithm converges under a properly designed optimality gap tolerance schedule. We then use RL to find the tolerance schedule that minimizes the total solution time by using a reward function that accounts for both the solution time per iteration and the improvement in the optimality gap. We use a mixed-integer economic model predictive control (MPC) case study to validate the proposed approach. Specifically, we consider the case where a reactor is used to manufacture multiple products and its operation is affected by disturbances, such as changes in product demands and inlet conditions. For a given disturbance, a mixed integer economic MPC problem is solved to determine the production sequence and dynamic transition profiles between products. The problem is solved using an inexact GBD algorithm, where the master problem determines the production sequence and transition times, and the subproblem considers the dynamic transitions. 
The results show that, without affecting the solution quality, the proposed RL-iGBD approach can lead to a reduction in solution time up to $27 \%$ on average compared to standard GBD, and $18 \%$ reduction in solution time compared to heuristic approaches for choosing the optimality gap tolerance. 

The remainder of the paper is organized as follows. In Section \ref{sec:igbd}, we introduce the classic GBD algorithm and the proposed iGBD algorithm. In Section \ref{sec:igbd-rl}, we introduce the generic framework of RL and discuss how the selection of the optimality gap tolerance in iGBD is addressed through RL. In Section \ref{sec:mimpc-model} we present the mixed integer MPC problem and the setup of the computational experiment. Finally, in Section \ref{sec:results} we present and analyze the results of the computational study.

\section{Inexact Generalized Benders Decomposition}\label{sec:igbd}

\subsection{Generalized Benders Decomposition}
We first introduce the classic GBD algorithm. Let's consider the following optimization problem.
\begin{equation}\label{eq:origin}
    \begin{aligned}
        \mathcal{P}(p):=\minimize \quad & f_1(y,p)+f_2(x,y,p)\\
        \subt \quad & g(y,p) \le 0 \\
        & h(x,y, p) \le 0 \\
        & y \in \mathbb{Z}^{n^y_d} \times \mathbb{R}^{n^y_c},x \in \mathbb{R}^{n^x_c}
    \end{aligned}
\end{equation}
where $p$ are the parameters of the problem, $n=n^y_d + n^y_c+n^x_c$ is the number of variables,and $f_1: \mathbb{Z}^{n^y_d} \times \mathbb{R}^{n_c^y} \rightarrow \mathbb{R}$, $f_2: \mathbb{Z}^{n^y_d} \times \mathbb{R}^{n_c^y} \times \mathbb{R}^{n^x_c} \rightarrow \mathbb{R}$, $g: \mathbb{Z}^{n^y_d} \times \mathbb{R}^{n_c^y} \rightarrow \mathbb{R}^{n^g}$, $h: \mathbb{Z}^{n^y_d} \times \mathbb{R}^{n_c^y} \times \mathbb{R}^{n^x_c} \rightarrow \mathbb{R}^{n^h}$. For this problem, if the variables $y$ are fixed, the resulting problem is a continuous optimization problem that usually can be solved faster than the original mixed integer problem. This structure can be utilized in a GBD approach, where the $y$ variables are the complicating variables and assigned to the master problem and the $x$ variables are assigned to the subproblem. Under this decomposition, the subproblem is parameterized by $y$ and $p$ and is equal to
\begin{equation}\label{eq:sp}
    \begin{aligned}
        S(y,p) := \minimize\quad & f_2(x,\bar y,p)\\
        \subt \quad & h(x,\bar{y}, p) \leq 0 \\
        & \bar{y}=y \quad : \lambda \\
        & \bar y \in \mathbb{R}^{n^y_d} \times \mathbb{R}^{n^y_c}, x \in \mathbb{R}^{n^x_c}
    \end{aligned}
\end{equation}
where $\lambda$ is the Lagrange multiplier vector for the equality constraint $\bar{y}=y$. We assume that the subproblem is always feasible for all possible values of the complicating variables $y$. The original problem can be reformulated as follows: 
\begin{equation}\label{eq:re-origin}
    \begin{aligned}
        \mathcal{P}(p):= \minimize \quad & f_1(y,p)+{S}(y,p)\\
        \subt \quad & g(y,{p}) \leq 0 \\
        & y \in \mathbb{Z}^{n^y_d} \times \mathbb{R}^{n^y_c}
    \end{aligned}
\end{equation}
Although this is an exact reformulation, its solution is challenging since the value function of the subproblem $S(y,p)$ is not known explicitly. The GBD algorithm iteratively constructs a convex piece-wise affine approximation of the value function using Benders cuts. For a given value of the complicating variables $\bar{y}=y$, the Benders optimality cut is
\begin{equation}\label{eq:benders-cut}
    \begin{aligned}
        {S}(y,p) \geq {S}(\bar{y},p) - \lambda^\top (y-\bar{y}). 
    \end{aligned}
\end{equation}
This approximation is exact when the value function $S(y,p)$ is convex. Using this approximation, the master problem can be reformulated as follows. 
\begin{equation} \label{eq: master all cuts}
    \begin{aligned}
        {M}_l(p) := \minimize \quad & f_1(y,p)+\eta\\
        \subt \quad & g(y,p) \leq 0 \\
        & \eta \geq {S}(y^{(l)},p) - (\lambda^{(l)})^\top(y-y^{(l)}) \quad \forall l\in \mathcal{L} \\
        & y \in \mathbb{Z}^{n^y_d} \times \mathbb{R}^{n^y_c}, \eta \in [\underline{\eta}, \overline{\eta}]
    \end{aligned}
\end{equation}
where $\eta$ is an auxiliary variable that approximates the value function ${S}(y,p)$ and the set $\mathcal{L}$ represents all the points used to approximate the value function. 
The solution of the reformulated master problem in Eq.~\ref{eq: master all cuts} is challenging, since the number of cuts that approximate the value function of the subproblem can be very large. For a given optimization problem $\mathcal{P}(p)$, GBD starts with a master problem without cuts and adds them iteratively as dictated by the solution of the master problem. Specifically, in iteration $l$, the master problem $M_l(p)$ is optimized to obtain $(y^{(l)}, \eta^{(l)})$. The complicating variables $y^{(l)}$ are fixed in the subproblem, which is then solved to obtain $x^{(l)}$, Lagrange multiplier $\lambda^{(l)}$, and ${S}({y}^{(l)},p)$. The Benders cut is constructed and added to the master problem, which is solved again. The objective value of the master problem $v^{(l)}=f_1(y^{(l)},p)+\eta^{(l)}$ provides a valid lower bound to the optimal objective value $v^*$ since the problem $M_{l}(p)$ is a relaxation of $\mathcal{P}(p)$ in Eq.~\ref{eq:re-origin}. The value $f_1(y^{(l)},p)+S(y^{(l)},p)$ corresponds to a feasible solution $(x^{(l)},y^{(l)})$ of problem $\mathcal{P}(p)$, therefore providing an upper bound. The algorithm terminates when the gap between the lower and upper bounds is less than a predetermined threshold $\varepsilon_{tol}$. Under the convexity of the value function $S(y,p)$, this procedure converges to the global optimal solution $(x^*,y^*)$ of the monolithic problem \cite{gbd}.

\subsection{Inexact Generalized Benders Decomposition}\label{subsec:igbd-alg}
In this section, we present the inexact GBD algorithm.
We define $\epsilon_{tol-MP}^{(l)}$ as the optimality gap tolerance for the master problem at iteration $l$. In this setting, the solution of the master problem stops once a feasible integer solution with an optimality gap $\epsilon_{MP}^{(l)}$ less than or equal to $\epsilon_{tol-MP}^{(l)}$, that is, $\epsilon_{MP}^{(l)}\leq \epsilon_{tol-MP}^{(l)}$, is found. We assume that the master problem is solved using standard branch-and-bound techniques. The feasible solution to the master problem $y^{(l)}$ is used to generate Benders cuts and obtain an upper bound $f_{1}(y^{(l)},p)+ S(y^{(l)},p)$. The best upper bound up to iteration $l$ is $UB^{(l)} = \min(UB^{(l-1)}, f_{1}(y^{(l)},p)+ S(y^{(l)},p))$. 

The feasible solution to the master problem is not guaranteed to provide a valid lower bound since the value of the associated objective function $v^{(l)}$ could be higher than the optimal value of the monolithic problem $v^*$. To overcome this limitation, we define $v^{(l)}(1-\epsilon_{MP}^{(l)})$ as the \textit{true lower bound}, and prove its validity in Proposition \ref{thm:valid_lower_bound}. 
We further force the monotonic increase of the true lower bound by updating $TLB^{(l)} = \max(TLB^{(l-1)}, v^{(l)}(1-\epsilon_{MP}^{(l)}))$. For a given feasible solution of the master problem and the corresponding solution of the subproblem, the incumbent optimality gap of iGBD is equal to $\varepsilon^{(l)}=(UB^{(l)}-TLB^{(l)})/UB^{(l)}$, referred to as \textit{Benders gap} in this paper. Consequently, we refer to $\epsilon_{MP}^{(l)}$ as the \textit{master gap}. The algorithm terminates once $\varepsilon^{(l)}< \varepsilon_{tol}$. Algorithm~\ref{agl:igbd} below summarizes the full iGBD procedure.

\begin{algorithm}[H]
\caption{iGBD algorithm.}
\begin{algorithmic}[1]
\label{agl:igbd}
  \STATE \textbf{Initialize:} $ILB^{(0)}=-\infty$, $TLB^{(0)}=-\infty$, $UB^{(0)}=\infty$, $l=0$, $\varepsilon^{(0)}=1$, $\varepsilon_{tol}$, $ \underline{\epsilon}_{tol-MP}$, $ \overline{\epsilon}_{tol-MP}$.
  \WHILE{$\varepsilon^{(l)}\ge\varepsilon_{tol}$}
    \STATE Set $\epsilon_{tol-MP}^{(l)} \in [\underline{\epsilon}_{tol-MP},\overline{\epsilon}_{tol-MP}] $.
    \STATE Update $l = l+1$.
    \STATE Solve the master problem and obtain $y^{(l)},\eta^{(l)},v^{(l)}, \epsilon_{MP}^{(l)}$.
    \STATE Update  $TLB^{(l)}=\max(TLB^{(l-1)}, v^{(l)}(1-\epsilon_{MP}^{(l)}))$.
    \STATE Solve subproblem for $\lambda^{(l)}$ and $S(y^{(l)},p)$. 
    \STATE Add Benders cut $\eta\ge S(y^{(l)},p) - (\lambda^{(l)})^\top(y-y^{(l)})$ to the master problem. 
    \STATE Update $UB^{(l)} = \min(UB^{(l-1)}, f_{1}(y^{(l)},p)+ S(y^{(l)},p))$.
    \STATE Update $\varepsilon^{(l)}=(UB^{(l)}-TLB^{(l)})/UB^{(l)}$.
  \ENDWHILE
\end{algorithmic}
\end{algorithm}


\subsection{Theoretical Properties of iGBD}\label{subsec:igbd-theory}
We shall first show in Proposition \ref{thm:valid_lower_bound} that $v^{(l)}(1-\epsilon_{MP}^{(l)})$ is a valid lower bound of the global optimal objective value $v^*$.
\begin{proposition}
    \label{thm:valid_lower_bound}
    At any iteration $l$, the value $v^{(l)}(1-\epsilon_{MP}^{(l)})$ is a valid lower bound on the optimal objective value $v^*$ for the problem, i.e. $v^{(l)}(1-\epsilon_{MP}^{(l)})\le v^*$. 
\end{proposition}
\begin{proof}
    Let $(v^{(l)})^*$ be the globally optimal value for the master problem in iteration $l$. It is known that $v^{(l)}(1-\epsilon_{MP}^{(l)}) \le (v^{(l)})^* \le v^{(l)}$. Since the master problem is a relaxed version of the original problem, $(v^{(l)})^*$ is always a valid lower bound of $v^*$, that is, $(v^{(l)})^* \le v^*$. Therefore, we obtain:
    \begin{equation*}
    v^{(l)}(1-\epsilon_{MP}^{(l)}) \le (v^{(l)})^* \le v^*
    \end{equation*}
\end{proof}

Proposition \ref{thm:valid_lower_bound} establishes that $v^{(l)}(1-\epsilon_{MP}^{(l)})$ is a valid lower bound, and the termination criterion is defined on this basis. We now turn to the convergence property of the proposed iGBD algorithm. The idea of an inexact solution of the master problem within the Benders decomposition framework was first proposed by \cite{geoffrion_multicommodity}, where convergence was argued by noting the finiteness of dual solutions in the linear subproblem. In contrast, the convergence property of the classic GBD algorithm is established rigorously (see Theorem 2.5 in \cite{gbd}). Here we extend the convergence argument of GBD to the inexact setting. Specifically, under the assumption that the optimality gap tolerance of the master problem asymptotically converges to zero, we show that the proposed iGBD algorithm achieves a finite convergence property. 

\begin{proposition}
    \label{thm:convergence}
    Assume that the complicating variable $y$ is in a compact set $Y$. With a tolerance sequence $\{\epsilon_{tol-MP}^{(l)}\}_{{l}\in\mathbb{N}}$ asymptotically converging to 0, the proposed iGBD algorithm terminates in a finite number of steps for any given $\varepsilon_{tol}>0$.
\end{proposition}

\begin{proof}
    Suppose that the algorithm does not terminate in a finite number of steps. Let $\{y^{(l)}, \eta^{(l)}\}$ be the sequence of master problem solutions. Following the assumption that the complicating variable $y$ is in a compact set, we may assume that a subsequence of $\{y^{(l)}\}$ converges to the point $\hat{y}$. Furthermore, we assume that a subsequence of the optimal Lagrange multiplier $\{\lambda^{(l)}\}$ converges to $\hat{\lambda}$ because the subproblem takes the stationary input $\hat{y}$. Consequently, the Benders cut inequality generated by the subproblem will converge to the form $\eta \ge S(\hat{y}) - \hat{\lambda}^\top(y-\hat{y})$, which must hold at point $(\hat{y}, \eta^{(l)})$ in the master problem, therefore:
    \begin{align*}
        & \eta^{(l)} \ge S(\hat{y}) - \bar\lambda^\top(\hat{y}-\hat{y}) = S(\hat{y}) \\
        \Leftrightarrow \quad & f_1(\hat{y})+\eta^{(l)} \ge f_1(\hat{y}) + S(\hat{y}) \\
        \Leftrightarrow \quad & v^{(l)} \ge UB^{(l)}
    \end{align*}
    leading to the following inequality: 
    \begin{align*}
        \varepsilon^{(l)} & = \frac{UB^{(l)} - TLB^{(l)}}{UB^{(l)}} \le 1-\frac{v^{(l)}(1-\epsilon_{MP}^{(l)})}{UB^{(l)}}\\
        & \le 1-\frac{v^{(l)}(1-\epsilon_{MP}^{(l)})}{v^{(l)}} = \epsilon_{MP}^{(l)} \le \epsilon_{tol-MP}^{(l)}
    \end{align*}
    
    Since $\epsilon_{tol-MP}^{(l)}$ asymptotically converges to 0, there exists a finite number $m$ such that $\varepsilon^{(m)}<\epsilon_{tol-MP}^{(m)}<\varepsilon_{tol}$, thus the algorithm will terminate, which completes the proof.
\end{proof}

\section{Reinforcement Learning-Enhanced Inexact Generalized Benders Decomposition}\label{sec:igbd-rl}

Although Algorithm \ref{agl:igbd} provides a general approach to address a computationally expensive master problem, selecting the master gap tolerance sequence $\epsilon_{tol-MP}^{(l)}$, which directly controls the convergence speed, constitutes a challenging black-box optimization problem. Here, we cast the iterative selection of the optimality gap tolerance as a sequential decision-making task. Specifically, we use RL to train a policy that adaptively selects $\epsilon_{tol-MP}^{(l)}$ based on both instance-specific features and the state of the solution.

\subsection{Reinforcement Learning}
Reinforcement learning (RL) is a computational framework for solving sequential decision-making problems through trial and error, guided by feedback in the form of a scalar reward \cite{sutton1998reinforcement}. In RL, the decision maker, known as the agent, interacts with an environment by observing its current state $s$, selecting an action $a$ according to a policy $\pi$, and receiving a reward $r$ that reflects the desirability of the chosen action.


 RL can be formalized through the Markov Decision Process (MDP) defined by the tuple $\mathcal{M} = (\mathcal{S}, \mathcal{A}, \mathcal{P}, \mathcal{R}, \gamma)$. The state space $\mathcal{S}$ represents all possible configurations of the environment, while the action space $\mathcal{A}$ defines the set of all possible actions. The transition model $\mathcal{P}: \mathcal{S} \times \mathcal{A} \times \mathcal{S} \rightarrow [0,1]$ specifies the probability $\mathcal{P}(s'|s,a)$ of transitioning to state $s'$ from state $s$ after taking action $a$. The reward function $\mathcal{R}: \mathcal{S} \times \mathcal{A} \rightarrow \mathbb{R}$ assigns a scalar value reflecting the immediate utility of a state–action pair. Finally, the discount factor $\gamma \in [0,1)$ weights future rewards relative to immediate ones.

The RL agent selects actions according to a policy $\pi \in \Pi: \mathcal{S} \times \mathcal{A} \rightarrow [0,1]$ which defines a conditional probability distribution on actions given states, that is, $a \sim \pi(a|s)$. The chosen action and the current state jointly determine the immediate reward $r$ and the next state $s' \sim \mathcal{P}(s' | s,a)$. The agent's objective is to maximize the expected cumulative discounted reward over time, which is denoted as $J(\pi)=\mathbb{E}_\pi \left[
    \sum^\infty_{t=0} \gamma^tr(s_t, a_t)
    \right]$. 

The solution approaches in RL are broadly categorized into value- and policy-based approaches. Value-based methods
estimate the expected return of each action given a state and select the action with the highest estimated value \cite{watkins1992q, mnih2015dqn, sutton1995sarsa}. Although effective in discrete action settings, these methods are less suited to problems involving large or continuous action spaces. Policy-based methods address this limitation by directly optimizing a parameterized policy using gradient ascent \cite{williams1992REINFORCE, schulman2015trpo, ppo}.
In practice, many RL algorithms combine these two paradigms in the form of actor–critic methods~\cite{konda1999actor_critic}, which jointly learn a value function (critic) and a policy (actor). This structure improves sample efficiency and stabilizes learning, making it particularly well suited for high-dimensional or complex tasks.

\subsection{Reinforcement Learning-Enhanced Gap Tolerance Selection}\label{subsec:igbd-rl}

In this section, we show how the gap tolerance selection task of the iGBD algorithm can be modeled as an MDP, which is then solved using RL algorithms. To make the notation consistent, the index of iteration $l$ is the same as the index of time steps in the environment. The environment considers the solution of a problem instance $\mathcal{P}(p)$ parameterized by $p$. Each episode terminates when either the algorithm converges or reaches the maximum iteration limit $T_{\max}$.

The state at iteration $l$ is defined as: 
\begin{equation}
    \begin{aligned}
        s_{l}= \left\{p, l, {\epsilon}_{tol-MP}^{(l)},\epsilon_{MP}^{(l)}, \varepsilon^{(l)}, \log(\varepsilon^{(l)}), \log\bigg(\frac{\varepsilon^{(l-1)}}{\varepsilon^{(l)}}\bigg)\right\}.
    \end{aligned}
\end{equation}

The state captures information on the parameters of the optimization problem $p$ and the
solution status of the iGBD algorithm. The iteration index $l$ is a direct signal of the progress of the algorithm and also reflects the growing size of the master problem as the Benders cuts cumulate. $\epsilon_{tol-MP}^{(l)}$ and $\epsilon_{MP}^{(l)}$ are the gap tolerance and the realized gap of the master problem, respectively. $\varepsilon^{(l)}$ is the incumbent Benders gap in the real-number scale, while $\log(\varepsilon^{(l)})$ presents the same gap in the logarithmic scale. Although these two variables are mathematically equivalent, the logarithmic scale provides a more visible change in later iterations when $\varepsilon^{(l)}$ is close to zero, and thus helps the agent identify the evolution of the system.
The last term $\log(\frac{\varepsilon^{(l-1)}}{\varepsilon^{(l)}})$ measures the relative improvement of the Benders gap in the logarithmic scale. Together, these features capture both the static problem context and the dynamic progress of the iGBD algorithm, providing the agent with the detailed information to effectively adjust the gap tolerance. 

The only action in this task is to determine the master gap tolerance $\epsilon_{tol-MP}^{(l)}$ at each iteration. 
To preserve the theoretical convergence property of iGBD, we design a parameterized gap tolerance selection region that adaptively shrinks as the Benders gap decreases, achieving the asymptotic convergence of the tolerance sequence.
Specifically, the policy outputs a normalized score $a_l \in [-1,1]$, which is mapped into the real action $\varepsilon_{MP}^{(l)}$ by: 
\begin{equation}\label{eq:action-map}
    \begin{aligned}
        \epsilon_{tol-MP}^{(l)} = \min\left(\underline{\epsilon}_{tol-MP}+\cfrac{a_l+1}{2}(\varepsilon^{(l-1)}-\underline{\epsilon}_{tol-MP}), \ \overline{\epsilon}_{tol-MP}\right)
    \end{aligned}
\end{equation}
so that the master gap tolerance is chosen below the incumbent Benders gap by design. A high score $a_l$ close to 1 yields a loose master gap tolerance, which saves runtime at a specific iteration, but potentially loses the ability to reduce the Benders gap. In contrast, a score $a_l$ close to -1 is more aggressive in reducing the gap at the cost of runtime. Figure~\ref{fig:action_space} provides an intuitive illustration of the adaptive gap selection region. Although $a_l$ is a continuous variable, we uniformly discretize the continuous space $[-1,1]$ into $K$ points to simplify learning. 

\begin{figure}[h!]
    \centering
    \includegraphics[
    width=\columnwidth,
    ]{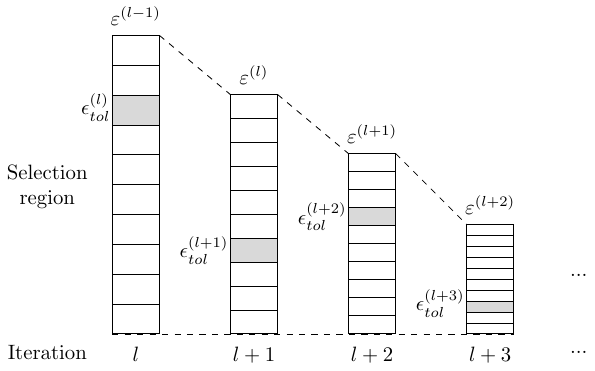}
    \caption{Intuitive illustration of the adaptive gap tolerance selection region.}
    \label{fig:action_space}
\end{figure}

After $\varepsilon_{MP}^{(l)}$ is determined by the policy, the master problem is solved for the values of the complicating variables, which are fixed in the subproblem to generate the Benders cut and obtain updated bounds. The state is then updated on the basis of the information from the solution, and the agent receives the reward signal that measures its performance. The reward design is crucial for learning efficiency. The goal of the RL agent is to minimize the overall runtime required for convergence. Let $t_{MP}^{(l)}$ be the runtime of the master problem in iteration $l$, then the reward $r_l$ is presented as:
\begin{equation}
    \begin{aligned}
        & r_1 = - t_{MP}^{(l)} / t_{ref} \\
        & r_2 = \alpha \ \log\left(\frac{\varepsilon^{(l-1)}}{\max(\varepsilon^{(l)},\underline{\epsilon}_{tol-MP})}\right) \\ 
        & r_l = r_1 + r_2 -1
    \end{aligned}
\end{equation}
where $\alpha$ is an adjustable weight coefficient and $t_{ref}$ is a predetermined reference master problem runtime. The first term $r_1$ directly penalizes the relative solution time for the master problem at each iteration, while the second term $r_2$ incentivizes the reduction of the Benders gap. The total reward is further subtracted by 1 to penalize the iteration number. 

In our experience, both reward components are necessary for effective policy learning. If only $r_1$ is included in the reward, the agent becomes myopically greedy in reducing the runtime per iteration by setting the optimality gap tolerance to large values. However, this can lead to an increase in the total solution time. Many recent works on the use of RL to enhance existing optimization algorithms exclusively consider $r_2$, focusing on reductions in optimality gaps \cite{tang_reinforcement_2020, chi_rl-cg_2023, yuan_rl-cg_2024}. Although this successfully reduces the overall runtime for their problems, it is not the case for the iGBD algorithm. Through computational experiments, we will demonstrate that gap improvement and solution time are sometimes conflicting metrics for the gap tolerance selection task in iGBD, and the agent must learn to balance these competing objectives via direct interactions with the environment. 
The policy is trained on randomly generated problem instances. After training, the policy $\pi_\theta$ is embedded in the iGBD algorithm to improve its performance. 
At the beginning of iteration $l$, the trained policy outputs the master gap tolerance ${\epsilon}_{tol-MP}^{(l)}$ given the state $s_l$. The RL-enhanced iGBD algorithm is presented in Algorithm~\ref{alg: RL-iGBD}.

We use the Proximal Policy Optimization (PPO) algorithm ~\cite{ppo} to train the RL agent due to its favorable trade-offs in stability, sample efficiency, and scalability. PPO is a policy gradient method that simultaneously estimates the policy (actor) and the value function (critic). In our setting, the critic estimates expected returns, that is, the total solution time, using an advantage function. The advantage is a measure of how much better a particular action (optimality gap tolerance) is compared to the average action in a given state. The actor then uses this advantage to guide policy updates. PPO collects trajectory data by interacting with the environment (iGBD) and updates both the policy and value networks using this experience. The policy is optimized by minimizing a surrogate loss function that reflects expected performance improvements while being regularized to ensure stable updates. A key step in PPO is the clipped objective function, which constrains policy updates and prevents large shifts from the current policy. Given its robustness and empirically demonstrated ability to produce stable learning in diverse environments, PPO is particularly well suited for the GBD setting considered in this work.

\begin{algorithm}[t]
\caption{RL-enhanced iGBD.}\label{alg: RL-iGBD}
\begin{algorithmic}[1]
  \STATE \textbf{Initialize:} $l=0$, $ILB^{(0)}=-\infty$, $TLB^{(0)}=-\infty$, $UB^{(0)}=\infty$, $\varepsilon_{tol}$, $\varepsilon^{(0)}=1$, $s_0=\{p, 0, \overline{\epsilon}_{tol-MP},\overline{\epsilon}_{tol-MP}, \varepsilon^{(0)}, \log(\varepsilon^{(0)}), 0\}$
  \WHILE{$\varepsilon^{(l)}\ge\varepsilon_{tol}$}
    \STATE Set $\epsilon_{tol-MP}^{(l)}=\pi_\theta(s_{l})$.
    \STATE Update $l = l+1$.
    \STATE Solve the master problem and obtain $y^{(l)},\eta^{(l)}, v^{(l)},\epsilon_{MP}^{(l)}$.
    \STATE Update $TLB^{(l)}=\max(TLB^{(l-1)}, v^{(l)}(1-\epsilon_{MP}^{(l)}))$.
    \STATE Solve subproblem for $\lambda^{(l)}$ and $S(y^{(l)},p)$. 
    \STATE Add Benders cut $\eta\ge S(y^{(l)},p) - (\lambda^{(l)})^\top(y-y^{(l)})$ to the master problem. 
    \STATE Update $UB^{(l)} = \min(UB^{(l-1)}, f_{1}(y^{(l)},p)+ S(y^{(l)},p))$.
    \STATE Update  $\varepsilon^{(l)}=(UB^{(l)}-TLB^{(l)})/UB^{(l)}$.
    \STATE Update  $s_{l}=\left\{p, l, {\epsilon}_{tol-MP}^{(l)},\epsilon_{MP}^{(l)}, \varepsilon^{(l)}, \log(\varepsilon^{(l)}), \log\left( \frac{\varepsilon^{(l-1)}}{\varepsilon^{(l)}}\right) \right\}$.
  \ENDWHILE
\end{algorithmic}
\end{algorithm}

\section{Case Study: Mixed Integer Economic MPC for Multi-Product System}\label{sec:mimpc-model}
\subsection{Model Description}
We consider a continuous production system in which $N_p$ products are manufactured on a fixed horizon $H$ in a single continuous stirred tank reactor (CSTR). We consider the case where the operation of the reactor is affected by changes in the inlet conditions and/or product demand. To account for such disturbances, a mixed integer economic model predictive control problem is solved to determine the optimal production sequence, production times, and transition profiles such that profit is maximized.

\subsubsection{Scheduling Model}
We define $\mathcal{I}_p=\{1, \dots , N_p\}$ as the set of products to be manufactured and the scheduling horizon is discretized into $N_{s}=N_{p}$ slots, with $\mathcal{I}_s=\{1, \dots , N_s\}$ being the set of slots. Each product is assigned into one slot exclusively, while each slot consists of two types of regimes: the production regime where a product is steadily manufactured, and the transition regime where the system state transits from the current operation steady state to that of another product. The first slot has three regimes: a transition regime that drives the intermediate state to the operating state of the product assigned in the first slot, a production regime for steady manufacturing, and another transition regime from the current operation state to that of the product assigned in the second slot. The last slot only considers one production regime. 

We define a binary variable $y_{ik}$ which is equal to 1 if product $i$ is assigned to the slot $k$ and zero otherwise. In addition, we define the binary variable $z_{ijk}$ which is equal to one if a transition occurs from the product $i$ to $j$ in the slot $k$ and zero otherwise.  Similarly, we define a binary variable $\hat{z_i}$ that is equal to 1 if a transition occurs from the intermediate state to the product $i$ in the first slot and zero otherwise. The transition logic constraints are presented below.
\begin{equation}\label{eq:logic}
    \begin{aligned}
        & \sum_{k\in \mathcal{I}_s} y_{ik} = 1 \quad \forall i\in \mathcal{I}_p \\
        & \sum_{i\in \mathcal{I}_p} y_{ik} = 1 \quad \forall k\in \mathcal{I}_s \\
        & \hat{z_i} = y_{i1} \quad \forall i \in \mathcal{I}_p \\
        & z_{ijk} \ge y_{ik} + y_{j,k+1}-1 \quad \forall  i\in \mathcal{I}_p, j\in \mathcal{I}_p, k\in \mathcal{I}_s \backslash \{N_s\} \\
        & z_{ijk} \le y_{ik} \quad \forall  i\in \mathcal{I}_p, j\in \mathcal{I}_p, k\in \mathcal{I}_s \backslash \{N_s\} \\
        & z_{ijk} \le y_{j,k+1} \quad \forall  i\in \mathcal{I}_p, j\in \mathcal{I}_p, k\in \mathcal{I}_s \backslash \{N_s\}
    \end{aligned}
\end{equation}
The first constraint enforces that each product is manufactured once and the second constraint guarantees that at each slot one product is manufactured. 

We define $\Theta_{ik}$ as the production time of product $i$ in slot $k$, $\theta_{k}^{t}$ is the transition time in slot k, $\theta_{ijk}$ as the transition time from $i$ to $j$ in slot $k$, and $\hat\theta_i$ as the transition time from the intermediate state to the operating state of product $i$. Moreover, $T^s_k$ and $T^e_k$ are the start and end times of the slot $k$, respectively. The timing constraints are the following.
\begin{equation}\label{eq:timing}
    \begin{aligned}
        & T^s_1 = 0 \\
        & T^e_{N_s} \le H \\
        & T^s_{k+1} = T^e_k \quad \forall k \in \mathcal{I}_s \backslash \{N_s\} \\
        & T^e_k = T^s_k + \sum_{i\in \mathcal{I}_p} \Theta_{ik} + \theta^t_k \quad \forall k\in \mathcal{I}_s \backslash \{N_s\} \\
        & T^e_{N_s} = T^s_{N_s} + \sum_{i\in \mathcal{I}_p} \Theta_{i{N_s}} \\
        & \Theta_{ik} \le y_{ik} H \quad \forall i\in \mathcal{I}_p,k\in \mathcal{I}_s \\
        & \theta^t_k = \sum_{i\in \mathcal{I}_p} \sum_{j\in \mathcal{I}_p} \theta_{ijk} z_{ijk} \quad \forall k \in \mathcal{I}_s \backslash \{1\} \\
        & \theta^t_1 = \sum_{i\in \mathcal{I}_p} \sum_{j\in \mathcal{I}_p} \theta_{ij1} z_{ij1} + \sum_{i\in \mathcal{I}_p} \hat{\theta}_i \hat{z}_i \\
        & \theta_{ijk} \ge \underline{\theta}_{ij} \quad \forall i\in \mathcal{I}_p , j \in \mathcal{I}_p, k \in \mathcal{I}_s \\
        & \hat{\theta}_i \ge \underline{\hat{\theta}}_i \quad \forall i \in \mathcal{I}_p
    \end{aligned}
\end{equation}
where $H$ is the scheduling horizon, $\underline{\theta}_{ij}$ is the minimum transition time from the product $i$ to $j$, and $\underline{\hat{\theta}}_i$ is the minimum transition time from the intermediate state to the steady state of the product $i$. Finally, we define $I_{ik}$ as the inventory of product $i$ in slot $k$, and $q_{ik}$ as the production amount of product $i$ in slot $k$. The inventory balance constraints are
\begin{equation}\label{eq:inv}
    \begin{aligned}
        & q_{ik} = r_i \Theta_{ik} \\
        & I_{ik} = I_{i, k-1} + q_{ik} - S_{ik} \quad \forall i\in \mathcal{I}_p, k \in \mathcal{I}_s \backslash \{1\}\\
        & I_{i1} = I_{i}^{0} + q_{i1} - S_{i1} \\
        & S_{iN_s} \geq d_i \quad \forall i \in \mathcal{I}_p
    \end{aligned}
\end{equation}
where $S_{ik}$ is the amount of product $i$ sold in slot k, $r_{i}$ is the production rate of product $i$, $I_{i}^{0}$ is the initial inventory of product $i$, and $d_{i}$ is the demand of product $i$, which is satisfied at the end of the scheduling horizon. 

\subsubsection{Dynamic Model}
The manufacturing system is an isothermal continuously stirred tank reactor where an irreversible third-order reaction $A \rightarrow 3B$ occurs, adopted from \cite{flores2006simultaneous}. The dynamic behavior is described by the following differential equation: 
\begin{align}
    \frac{dc(t)}{dt} = \frac{F(t)}{V} (c_f - c(t)) - kc(t)^3
\end{align}
where $c(t)$ is the concentration of the reactant and $F(t)$ is the inlet flowrate. The parameters are the composition of the feed stream $c_f=1 \ \text{mol}/\text{L}$, the reactor volume $V=5000\text{L}$, and the reaction constant $k=2\text{L}^2/\text{mol}^2 h$. The reactor can manufacture multiple products by adjusting the inlet flowrate. We define as $c^{ss}_i$ the steady state concentration and as $F^{ss}_i$ the steady state flowrate of product $i$.

We consider all possible transitions simultaneously and define $c_{ijk}(t)$, $F_{ijk}(t)$ as the value of the concentration and flowrate at time $t$ during a transition from product $i$ to $j$ in the slot $k$. The equations that describe the behavior of the reactor for a transition are
\begin{equation*}\label{eq:dae1}
    \begin{aligned}
        & \frac{dc_{ijk}(t)}{dt} = \frac{F_{ijk}(t)}{V} (c_f - c_{ijk}(t)) - k c_{ijk}(t)^3 \\
        & c_{ijk}(0) = c^{ss}_i \\ 
        & c_{ijk}(\theta_{ijk}) = c^{ss}_j \\ 
        & F_{ijk}(0) = F^{ss}_i \\ 
        & F_{ijk}(\theta_{ijk}) = F^{ss}_j \\
        & c^L \leq c_{ijk}(t) \leq c^U\\
        & F^L \leq F_{ijk}(t) \leq F^U,
    \end{aligned}
\end{equation*}
where $c^L$, $c^U$, $F^L$, $F^U$ are the lower and upper bounds of $c_{ijk}$ and $F_{ijk}$, respectively. To account for the fact that the transition time is an optimization variable, we define the scaled time $\tau = t/\theta_{ijk}^{t}$ and reformulate the above differential equations as follows.
\begin{equation}\label{eq:dae1}
    \begin{aligned}
        & \frac{dc_{ijk}(\tau)}{d\tau} = \left(\cfrac{F_{ijk}(\tau)}{V} (c_f - c_{ijk}(\tau)) - kc_{ijk}(\tau)^3 \right) \theta_{ijk} \\
        & c_{ijk}(0) = c^{ss}_i \\
        & c_{ijk}(1) = c^{ss}_j \\
        & F_{ijk}(0) = F^{ss}_i \\
        & F_{ijk}(1) = F^{ss}_j \\
        & c^L \le c_{ij}(\tau) \le c^U \\
        & F^L \le F_{ij}(\tau) \le F^U 
    \end{aligned}
\end{equation}

Similarly, for a given $c_0$, the dynamic transition from the intermediate state to the steady operating state of product $i$ in the first slot can be described as the following constraints: 
\begin{equation}\label{eq:dae2}
    \begin{aligned} 
        & \cfrac{d\hat c_{i}(\tau)}{d\tau} = \left(\cfrac{\hat F_{i}(\tau)}{V} (c_f - \hat c_{i}(\tau)) - \hat kc_{i}(\tau)^3 \right)  \hat\theta_{i} \\
        & \hat c_{i}(0) = c_0 \\
        & \hat c_{i}(1) = c_i \\
        & \hat F_{i}(1) = F^{ss}_i \\
        & \hat c^L \le \hat c_{i}(\tau) \le \hat c^U \\
        & \hat F^L \le \hat F_{i}(\tau) \le \hat F^U \\
    \end{aligned}
\end{equation}
where
$\hat c^L$, $\hat c^U$, $\hat F^L$, $\hat F^U$ are the lower and upper bounds of $\hat c_{i}$ and $\hat F_{i}$, respectively.

\subsubsection{Mixed-Integer Economic MPC Problem}

The goal of the integrated problem is to find the optimal production schedule and dynamic operation such that profit is maximized. The objective has two components. The first, $\Phi_{1}$, considers sales minus inventory, operating, and fixed transition costs. This term is equal to
\begin{equation}\label{eq:obj-schedule}
    \begin{aligned}
        \Phi_{1} = \sum_{i\in \mathcal{I}_p} \sum_{k \in \mathcal{I}_s} (Pr_{i}S_{ik} - C^{op}_{i}q_{ik}) - C^{inv}I_{ik} - \sum_{i \in \mathcal{I}_p} \sum_{j \in \mathcal{I}_p}\sum_{k \in \mathcal{I}_s} C^{tr}_{ij} z_{ijk}
    \end{aligned}
\end{equation}
where $Pr_{i}$ and $C_{i}^{op}$ are the sale price and operating cost of product $i$, respectively, $C^{inv}$ is the inventory cost, and $C^{tr}_{ij}$ is the fixed transition cost for the dynamic transition from product $i$ to $j$ in slot $k$.

The second term in the objective is related to the variable transition cost between the products, i.e. it depends on the transition time. For a transition from product $i$ to $j$ in slot $k$ with transition time $\theta_{ijk}$ this cost is equal to
\begin{equation*}
    \Phi_{ijk} = \alpha_u \theta_{ijk} \int^{1}_{0} (F_{ijk}(\tau)-F^{ss}_j)^2 d\tau
\end{equation*}
and the cost from the intermediate state $c_{0}$ to the product $i$ is
\begin{equation}
    \hat\Phi_{i} = \alpha_u \hat{\theta}_{i} \int^{1}_{0} (\hat{F}_{i}(\tau)-F^{ss}_i)^2 d\tau.
\end{equation}

In general, the optimization problem is presented below. 
\begin{equation}
    \begin{aligned}
        \mathcal{P}(p):= \text{minimize} & \quad \Phi_{1} - \sum_{i \in \mathcal{I}_p} \sum_{j \in \mathcal{I}_p}\sum_{k \in \mathcal{I}_s}\Phi_{ijk} z_{ijk} - \sum_{i \in \mathcal{I}_p} \hat \Phi_{i} \hat{z}_{i} \\
        \text{subject to} & \quad \text{Eqs. (\ref{eq:logic}), (\ref{eq:timing}), (\ref{eq:inv}), (\ref{eq:obj-schedule})}.\\
        & \quad \text{Eqs.~(\ref{eq:dae1})} \quad \forall i \in \mathcal{I}_p, j\in \mathcal{I}_p, k\in \mathcal{I}_s \\
        & \quad \text{Eqs.~(\ref{eq:dae2})} \quad \forall i \in \mathcal{I}_p
    \end{aligned}
\end{equation}

This is a Mixed Integer Dynamic Optimization (MIDO) problem which is transformed into a Mixed Integer Nonlinear Programming problem by discretizing the differential equations and integrals. 

\subsubsection{Decomposition-Based Solution Approach}
If the scheduling-related variables are fixed ($\Theta_{ik}$, $\theta_{k}^{t}$, $\theta_{ijk}$, $\hat\theta_i$, $T^s_k$, $T^e_k$,$I_{ik}$, $q_{ik}$, $S_{ik}$, $y_{ik}$, $z_{ijk}$, $\hat{z_i}$), the mixed integer MPC problem reduces to multiple independent subproblems. Each subproblem considers either a transition between two products in a slot or a transition from the intermediate state to the product manufactured in the first slot. We apply the GBD-based algorithm to solve this problem. The subproblem determining the transition cost from the product $i$ to $j$ in the slot $k$ is defined as:
\begin{equation}
    \begin{aligned}
        {S}_{ijk}(\theta_{ijk}) := \text{minimize} \quad & \Phi_{ijk}\\
        \text{subject to} \quad & \text{Eqs. (\ref{eq:dae1})}\\
        & \bar{\theta}_{ijk} = \theta_{ijk}: \lambda_{ijk}
    \end{aligned}
\end{equation}
where $\lambda_{ijk}$ is the Lagrange multiplier associated with the equality constraint $\bar{\theta}_{ijk} = \theta_{ijk}$. Similarly, the subproblem determining the transition cost from intermediate state to operating state of product $i$ is the following:
\begin{equation}
    \begin{aligned}
        \hat {{S}}_{i}(\hat\theta_{i}) :=  \text{minimize} \quad & \hat\Phi_{i}\\
        \text{subject to} \quad & \text{Eqs. (\ref{eq:dae2})}\\
        & \bar{\hat\theta}_{i} = \hat\theta_{i}: \hat\lambda_i
    \end{aligned}
\end{equation}
The master problem determines the production sequence, production and transition time while using an approximation of the transition costs via the Benders cuts.  We introduce auxiliary variables $\eta_{ijk}$ and $\hat\eta_i$ to approximate the cost of each dynamic transition. The master problem is formulated as follows. 
\begin{equation}
    \begin{aligned}
        \text{minimize} \quad & \Phi - \sum_{i \in \mathcal{I}_p} \sum_{j \in \mathcal{I}_p}\sum_{k \in \mathcal{I}_s}\eta_{ijk}z_{ijk} - \sum_{i \in \mathcal{I}_p} \hat \eta_i \hat z_{i} \\
        \text{subject to}\quad & \text{Eqs. (\ref{eq:logic}), (\ref{eq:timing}), (\ref{eq:inv}), (\ref{eq:obj-schedule})} \\
        & \eta_{ijk} \ge {S}_{ijk}(\theta^{(l)}_{ijk})-\lambda^{(l)}_{ijk} (\theta_{ijk} - \theta^{(l)}_{ijk}) \quad \forall i , j\in \mathcal{I}_p, l \in \mathcal{L}\\
        & \hat\eta_{i} \ge \hat{{S}}_{i}(\hat\theta^{(l)}_{i})- \hat \lambda^{(l)}_{i} (\hat \theta_{i} - \hat \theta^{(l)}_{i}) \quad \forall i \in \mathcal{I}_p, l \in \mathcal{L}
    \end{aligned}
\end{equation}
where $l$ is the iteration index. All bilinear terms ($\eta_{ijk}z_{ijk}$, $\hat \eta_i \hat z_{i}$, $\theta_{ijk}z_{ijk}$, $\hat \theta_i \hat z_{i}$) in the master problem are linearized using the McCormick envelopes, leading to an MILP problem \cite{mitrai2022multicut}. 

\subsection{Experimental Setup}\label{sec:setup}


\subsubsection{Optimization Model Configurations}
To fully validate the proposed method, we train and test the agent on problems with a varying number of products. Each problem instance is characterized by randomly generated parameters $p=\{c_0, \{d_i\}_{i\in \mathcal{I}_p}\}$, where $c_0\sim\mathcal{U}(0.8,1.2)$ is the inlet flow concentration and $d_i$ is the demand of product $i$. For each problem instance, the product demand $d_i$ is generated randomly following a uniform distribution, $d_i\sim \mathcal{U}(0.9d_i^\text{nom}, 1.1d_i^\text{nom})$ with $d_i^\text{nom}$ being the nominal demand. Other relevant parameters for each case are provided in the Supplementary Material. The randomly generated parameter values should guarantee the feasibility of the master problem as discussed in \cite{mitrai_computationally_2024}. Other model parameters include the scheduling horizon $H=168 \text{h}$, inventory cost coefficient $C^{inv}=0.5 \$/\text{mol}$, control cost coefficient $\alpha_u=0.5$ for each transition. The MILP master problem is formulated in Pyomo \cite{pyomo} and solved with Gurobi (version 12.0.1) \cite{gurobi}. The dynamic optimization subproblems are modeled in Pyomo.DAE \cite{pyomo_dae}, discretized via the backward Euler method with 120 finite elements, and solved using IPOPT (version 3.14.17) \cite{ipopt}. The GBD algorithm termination threshold is set as $\varepsilon_{tol}=1\times10^{-3}$. 

\subsubsection{Reinforcement Learning Implementation}

The maximum number of iterations is $T_{\max}=50$ for each episode. For the parameter values in the reward function, we set the weight coefficient $\alpha=2$ and the reference runtime $t_{ref}$ to 0.2, 0.45, and 0.75 s for the case of five, six, and seven products, respectively. The upper and lower bounds of the master gap tolerance are set as $\overline{\epsilon}_{tol-MP}=0.3$ and $\underline{\epsilon}_{tol-MP}=1\times10^{-3}$, respectively. 

The PPO algorithm is implemented using Stable Baselines3 \cite{stable-baselines3}, a PyTorch-based RL library. As mentioned in Section~\ref{subsec:igbd-rl}, we consider a normalized score $a_l\in[-1,1]$ as the output of the policy network and map it to $\epsilon_{tol-MP}^{(l)}$ using Eq.~(\ref{eq:action-map}). The continuous domain is uniformly discretized into $K=11$ points, leading to a discrete action space $\mathcal{A} =\bigl\{-1,-0.8,\dots,1\bigr\}$. For the PPO algorithm, the \(K\)-dimensional multi-discrete action space is factorized into \(K\) independent categorical distributions. At each step \(l\), the policy network produces a logit for each category based on the state vector $s_l$, passes all logits through a softmax to define a categorical distribution and samples from the distribution as the output $a_l$. During training, the action is randomly sampled to promote exploration, while during testing, the action with maximal logit is selected deterministically. Both the actor and the critic are feedforward neural networks. The actor network has 3 layers with 64 nodes per layer, while the critic network has 2 layers with 64 nodes each. For the hyperparameter values of PPO, the discount factor is $0.99$, the clip range is $0.15$, while the remaining parameters are set as default. During training, the learning rate is $5\times10^{-4}$, the number of steps to run per update is 512, and the batch size is 128. We train a separate agent for each case, each for 20000 training time steps. All experiments were performed on an Apple MacBook Air equipped with an M2 8-core chip and 8 GB of RAM.

\subsubsection{Benchmarks and Evaluation Metrics}
To assess the performance of iGBD with the RL-based tolerance selection policy (RL-iGBD), we consider the following three baseline policies: 

\begin{itemize}
    \item \textbf{Rand-iGBD} ($a_l\sim\mathcal{U}(\mathcal{A})$): iGBD with uniformly random outputs from the adaptive action space, mimicking an untrained RL policy. We choose this baseline to filter out the influences from the adaptive action space, demonstrating the effectiveness of training.
    
    \item \textbf{Exp-iGBD} ($\epsilon_{tol-MP}^{(l)}=\max (\overline{\epsilon}_{tol-MP}\cdot \alpha_{exp}^{l-1}, \underline{\epsilon}_{tol-MP})$): iGBD with an open-loop scheduled exponential decay sequence. The exponential factor is set as  $\alpha_{exp}=0.8$.
    
    \item \textbf{Classic GBD} ($\epsilon_{tol-MP}^{(l)}=\underline{\epsilon}_{tol-MP}$): Optimal solution of the master problem at every iteration $l$. 
\end{itemize}

In the computational experiments, all runtimes reported correspond to CPU time, as returned directly by the solvers. Our case study involves one master problem and a number $N_p$ of subproblems in each iteration. All subproblems can be solved independently in parallel. As metrics for evaluating performance, we define the cumulative CPU solution time to solve the master problem $t_{MP}$:
\begin{align}
    t_{MP} = \sum^T_{l=1} t^{(l)}_{MP}
\end{align}
and the cumulative total CPU solution time $t_{total}$: 
\begin{align}
    t_{total} = \sum^T_{l=1}\left(
    t^{(l)}_{MP} + \max \left\{\hat{t}_{SP}^{(l)}, t_{SP,1}^{(l)},\dots,t_{SP,N_p-1}^{(l)}\right\}
    \right)
\end{align}
where $T$ is the length of the episode. In iteration $l$, $\hat{t}_{SP}^{(l)}$ is the CPU time to solve the transition subproblem from the intermediate state in slot 1, and $t_{SP, k}^{(l)}$ is the CPU time to solve the transition sub-problem between products in slot $k$. 

\section{Computational Results}\label{sec:results}

\subsection{Performance Comparison} \label{subsec:episode}

To illustrate the impact of the RL-based and baseline tolerance selection policies on algorithm convergence, we analyze their performances in solving a representative problem instance. We consider the case of five products with parameters $c_0 = 1.0$, $d_A = 3000$, $d_B=5000$, $d_C=6000$, $d_D=10000$, $d_E=14000$. Each parameter value is selected as the midpoint of its sample range. All methods converge to the same optimal solution with objective value equal to $1.089\times 10^7$. 

\begin{figure}[h!]
    \centering
    \includegraphics[width=\columnwidth,height=0.85\textheight,keepaspectratio]{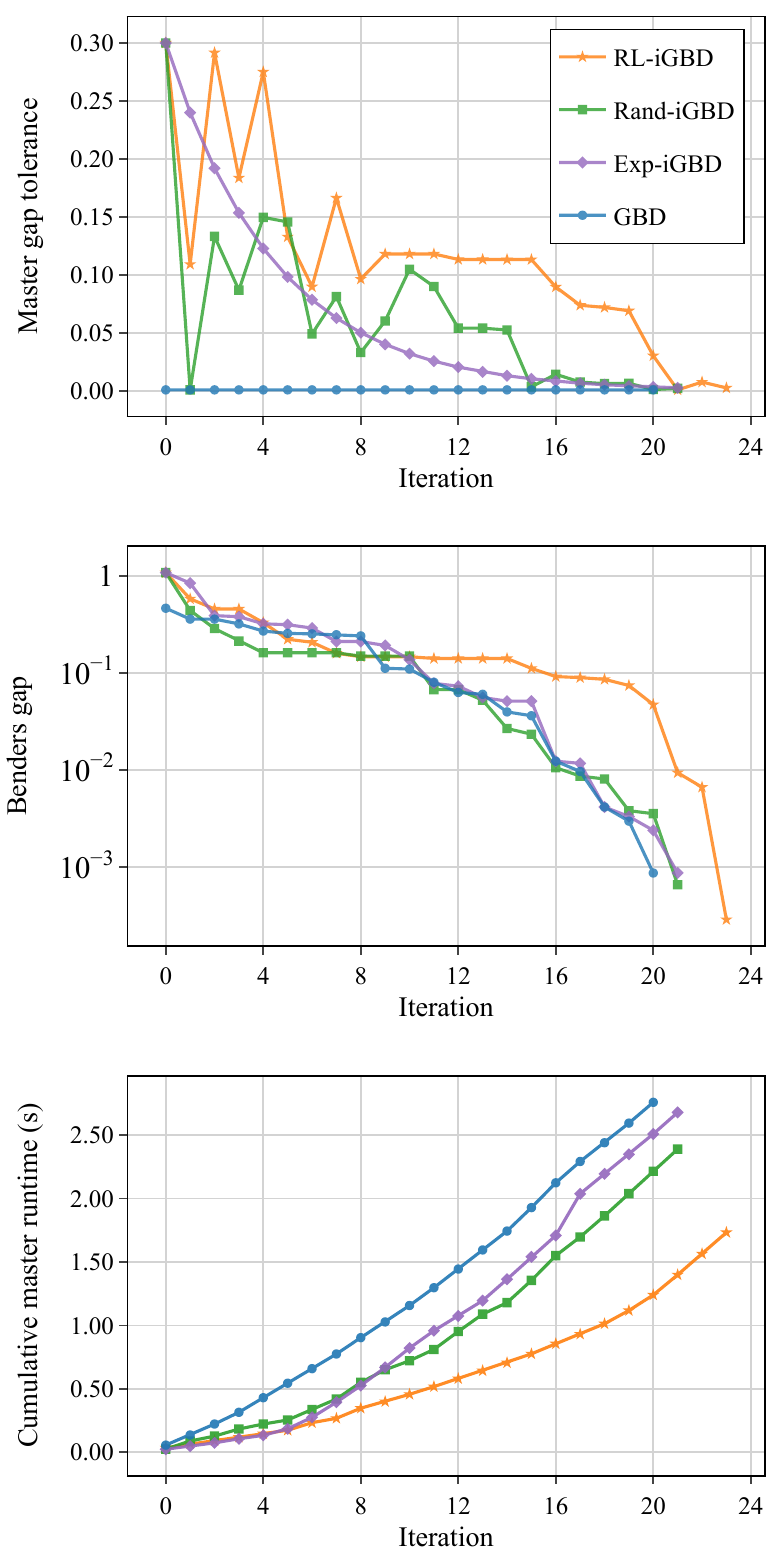}
    \caption{Per-iteration performance of different methods: master gap tolerance (top figure), Benders gap (medium figure) and cumulative master runtime (bottom figure).}
    \label{fig:single-convergence}
\end{figure}

Fig~\ref{fig:single-convergence} shows the convergence profiles for each method. The proposed RL-iGBD requires 23 iterations to reach convergence, whereas Rand-iGBD and Exp-iGBD each require 21 iterations, and the classic GBD requires 20 iterations. 
The top figure shows the master gap tolerances $\epsilon_{tol-MP}^{(l)}$ selected at each iteration. The classic GBD solves the master problem optimally at every iteration, with its tolerance always being $\underline{\epsilon}_{tol-MP}$.
All inexact methods continuously adjust the tolerance during the solve and eventually drive it toward $\underline{\epsilon}_{tol-MP}$. The cumulative master problem runtime of RL-iGBD, Rand-iGBD, Exp-iGBD and classic GBD are 1.73s, 2.38s, 2.68s, and 2.75s, respectively, as shown in the bottom figure. All three inexact strategies outperform the classic GBD implementation in terms of cumulative master problem CPU time, while RL-iGBD yields the largest improvement of 37.1\% over the classic GBD. 

It is interesting to observe the behavior of the RL-iGBD method compared with the baselines. During the early iterations without too many Benders cuts added, as the master problem has a relatively small scale and is easy to optimize, RL-iGBD selectively tightens the tolerance to reduce the Benders gap, and it is hard to distinguish it from other methods in the middle panel. 
At iteration 10, RL-iGBD starts to maintain a high tolerance plateau, which slows the reduction of the Benders gap and ultimately leads to more iterations to converge. However, staying at a high gap tolerance reduces the runtime per iteration much more compared with other methods. Eventually, although RL-iGBD takes 4 extra iterations to converge, it achieves a reduction of cumulative runtime, which aligns with our training goal. These findings demonstrate that the RL policy effectively balances convergence speed and computational efficiency. 

\subsection{Statistical Performance} 

\begin{table*}
  \centering
  \caption{Performance metrics of RL-iGBD and baseline methods across varying product numbers.}
  \label{tab:solution-time-stats}
  \makebox[\textwidth][c]{%
    \begin{tabular}{@{}c c c c c c c@{}}
      \toprule
      \# Product 
        & Method 
        & \makecell{Median\\\# iterations}
        & \makecell{Avg. total\\CPU time (s)}
        & \makecell{Std. total\\CPU time (s)} 
        & \makecell{Avg. master\\CPU time (s)}
        & \makecell{Std. master\\CPU time (s)} \\
      \midrule
      \multirow{4}{*}{$N_p=5$}
        & \textbf{RL-iGBD} & \textbf{24} & \textbf{3.49} & $\pm${0.160} & \textbf{1.82} & {$\pm$0.097} \\
            & Rand-iGBD    &   22   & 3.97 & $\pm$0.282 &  2.35  &  $\pm$0.230  \\
            & Exp-iGBD     &   22   & 4.20 & $\pm$0.268 &  2.64  &  $\pm$0.190  \\
            & GBD          &   21   & 4.56 & $\pm$0.305 &  2.94  &  $\pm$0.211 \\
      \midrule
      \multirow{4}{*}{$N_p=6$}
      & \textbf{RL-iGBD}     &   \textbf{25} & \textbf{5.94} & $\pm${0.692}  &   \textbf{4.13}  &  {$\pm$0.663} \\
            & Rand-iGBD   &   22   & 7.29 & $\pm$0.430 &  5.65  &  $\pm$0.405 \\
            & Exp-iGBD    &   21   & 7.45 & $\pm$0.299 &  5.84  &  $\pm$0.257 \\
            & GBD         &   21   & 8.13 & $\pm$0.239 &  6.54  &  $\pm$0.200 \\
    \midrule
    \multirow{4}{*}{$N_p=7$}
      & \textbf{RL-iGBD}     &   \textbf{25}  & \textbf{10.89} & $\pm${1.554} &   \textbf{8.94}  &  {$\pm$1.479}\\
            & Rand-iGBD   &   20  & 13.32 & $\pm$0.986 &   11.62  &  $\pm$0.915\\
            & Exp-iGBD    &   20  & 13.81 & $\pm$1.164 &   12.14  &  $\pm$1.108\\
            & GBD         &   19  & 14.09 & $\pm$1.106 &   12.56  &  $\pm$1.020 \\
      \bottomrule
    \end{tabular}%
  }
\end{table*}

\begin{figure*}
    \centering
    \includegraphics[width=1.0\textwidth]{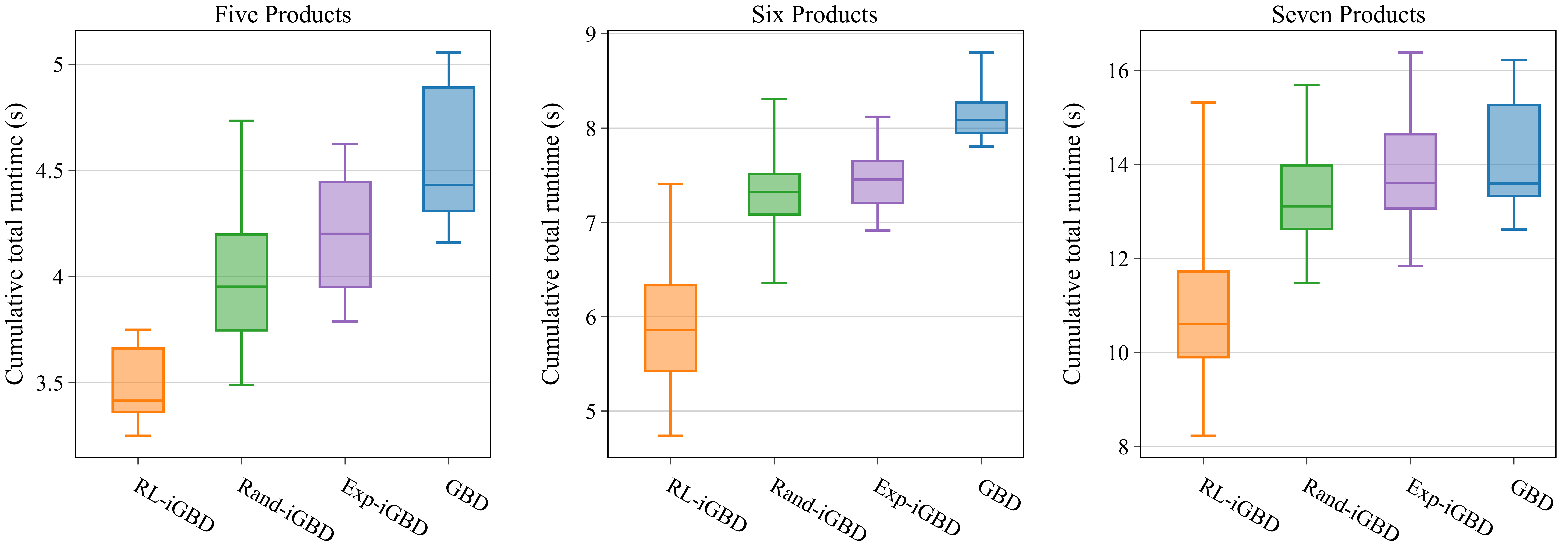}
    \caption{Box‐plot of cumulative total runtimes for different methods across five- (left figure), six- (middle figure), and seven-product (right figure) cases.} 
    \label{fig:50-boxplot-total}
\end{figure*}

\begin{figure*}
    \centering
    \includegraphics[width=1.0\textwidth]{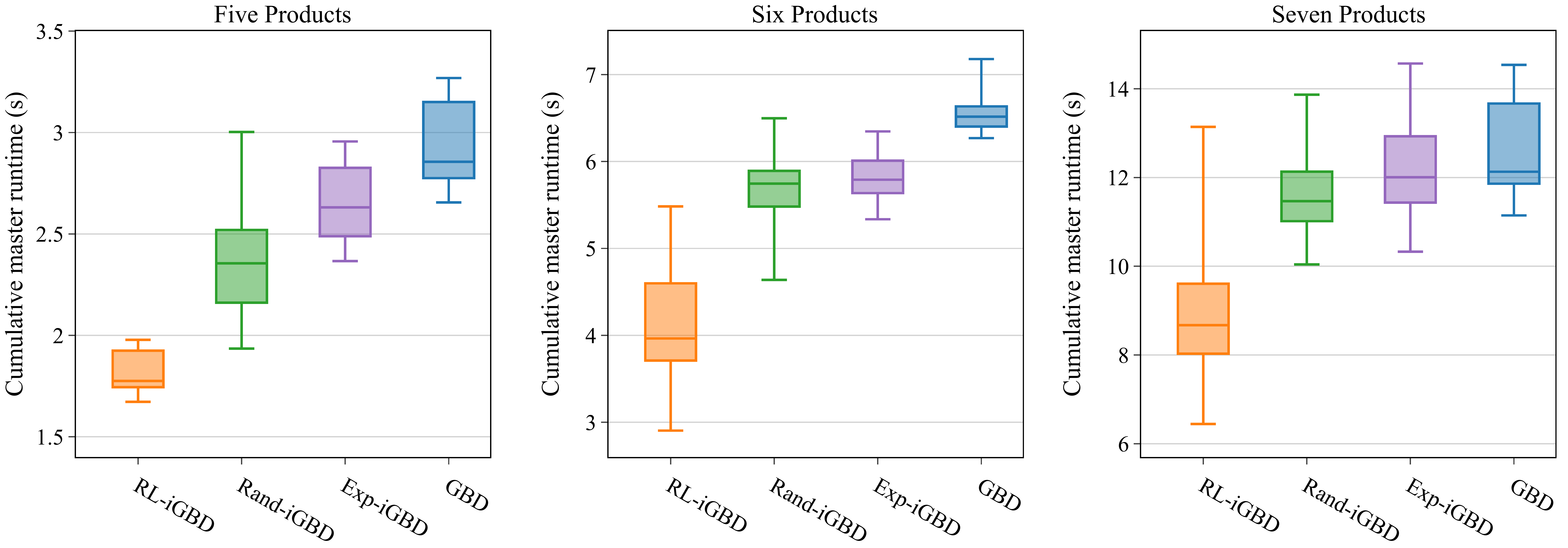}
    \caption{Box‐plot of cumulative master problem runtimes for different methods across five- (left figure), six- (middle figure), and seven-product (right figure) cases.} 
    \label{fig:50-boxplot}
\end{figure*}

To comprehensively assess the performance of the RL-iGBD method, we generate 50 random instances for each case. All instances are solved before the maximum iteration limit $T_{\max}$. 

For each method, Table~\ref{tab:solution-time-stats} summarizes the median number of iterations, the total solution runtime, and the cumulative solution runtime of the master problem. We observe that although all inexact approaches require more iterations than the standard GBD, the total solution time is lower. 
However, for the random and exponential optimality gap tolerance selection policy, the benefits diminish since, for seven products, the average solution time is similar to that of standard GBD (see Fig~\ref{fig:50-boxplot-total} and Fig~\ref{fig:50-boxplot}). However, for the proposed RL-iGBD algorithm, although the solution time increases as the number of products increases, the average solution time is lower than the other inexact approaches and standard GBD. Specifically, compared to standard GBD, the proposed RL approach achieves a reduction in total solution time of 23.4\%, 26.9\% and 22.7\% on average, for the cases of five-, six-, and seven-product cases, respectively. Moreover, the RL approach is 17\% and 20\% faster than the Rand-iGBD and Exp-iGBD approaches for seven products. Finally, we observe that as the number of products increases, the variance in the solution time for all methods increases. This can be attributed to the greater complexity of the learning task in higher-dimensional problems as well as the greater variance in the complexity of the optimization problem itself. 

\section{Conclusion}\label{sec:concl}
In this paper, we have introduced an iGBD framework to alleviate the computational burden of solving challenging master problems. Instead of solving the master problem optimally at every iteration, iGBD accepts any feasible solution within a predefined optimality gap tolerance. Moreover, we formulate tolerance selection as a sequential decision-making task, leveraging reinforcement learning to train a parameterized policy that dynamically balances runtime reduction and optimality gap improvement. Computational experiments on an economic mixed integer MPC problem demonstrate the computational efficiency of the proposed RL-iGBD method compared to classic GBD and inexact baselines tuned heuristically. 

Future work will evaluate the proposed RL-iGBD method on a broader class of real-world, large-scale problems to validate its robustness under various uncertainty regimes. The core concept of inexact master-problem tolerances could also be generalized to other decomposition frameworks (e.g., column generation or Lagrangian decomposition). From the RL perspective, leveraging sequence-based network architectures (e.g., recursive neural networks or transformers) to exploit complete solution trajectories may enhance sample efficiency and policy stability. Finally, exploring meta-learning for rapid policy transfer across different problems would further strengthen the adaptability of the framework.

\section*{Acknowledgements}

The partial financial support of NSF CBET (award number 2313289) is gratefully acknowledged. IM would like to acknowledge financial support from the McKetta Department of Chemical Engineering.

\section*{Data Availability}
Codes for computational experiments are published at \url{https://github.com/BambooSticker/RL-iGBD}.

\bibliographystyle{elsarticle-num}
\bibliography{references}

\begin{thebibliography}{10}
\expandafter\ifx\csname url\endcsname\relax
  \def\url#1{\texttt{#1}}\fi
\expandafter\ifx\csname urlprefix\endcsname\relax\def\urlprefix{URL }\fi
\expandafter\ifx\csname href\endcsname\relax
  \def\href#1#2{#2} \def\path#1{#1}\fi

\bibitem{benders1962partitioning}
J.~F. Benders, Partitioning procedures for solving mixed-variables programming problems, Numerische Mathematik 4 (1962) 238--252.

\bibitem{gbd}
A.~M. Geoffrion, Generalized benders decomposition, Journal of optimization theory and applications 10 (1972) 237--260.

\bibitem{elcci2022stochastic}
{\"O}.~El{\c{c}}i, J.~Hooker, Stochastic planning and scheduling with logic-based benders decomposition, INFORMS Journal on Computing 34~(5) (2022) 2428--2442.

\bibitem{laporte1993integer}
G.~Laporte, F.~V. Louveaux, The integer {L}-shaped method for stochastic integer programs with complete recourse, Operations research letters 13~(3) (1993) 133--142.

\bibitem{you2013multicut}
F.~You, I.~E. Grossmann, Multicut benders decomposition algorithm for process supply chain planning under uncertainty, Annals of Operations Research 210 (2013) 191--211.

\bibitem{luo2024design}
Y.~Luo, M.~Ierapetritou, Design and operation of modular biorefinery supply chain under uncertainty using generalized benders decomposition, AIChE Journal 70~(8) (2024) e18458.

\bibitem{lohmann2017tailored}
T.~Lohmann, S.~Rebennack, Tailored benders decomposition for a long-term power expansion model with short-term demand response, Management Science 63~(6) (2017) 2027--2048.

\bibitem{lara2018deterministic}
C.~L. Lara, D.~S. Mallapragada, D.~J. Papageorgiou, A.~Venkatesh, I.~E. Grossmann, Deterministic electric power infrastructure planning: Mixed-integer programming model and nested decomposition algorithm, European Journal of Operational Research 271~(3) (2018) 1037--1054.

\bibitem{chu2013integrated}
Y.~Chu, F.~You, Integrated scheduling and dynamic optimization of complex batch processes with general network structure using a generalized benders decomposition approach, Industrial \& Engineering Chemistry Research 52~(23) (2013) 7867--7885.

\bibitem{mitrai2022multicut}
I.~Mitrai, P.~Daoutidis, A multicut generalized benders decomposition approach for the integration of process operations and dynamic optimization for continuous systems, Computers \& Chemical Engineering 164 (2022) 107859.

\bibitem{menta2020learning}
S.~Menta, J.~Warrington, J.~Lygeros, M.~Morari, Learning solutions to hybrid control problems using benders cuts, in: Learning for Dynamics and Control, PMLR, 2020, pp. 118--126.

\bibitem{bansal2003new}
V.~Bansal, V.~Sakizlis, R.~Ross, J.~D. Perkins, E.~N. Pistikopoulos, New algorithms for mixed-integer dynamic optimization, Computers \& Chemical Engineering 27~(5) (2003) 647--668.

\bibitem{aytug2015feature}
H.~Aytug, Feature selection for support vector machines using generalized benders decomposition, European Journal of Operational Research 244~(1) (2015) 210--218.

\bibitem{santana2025support}
{\'I}.~Santana, B.~Serrano, M.~Schiffer, T.~Vidal, Support vector machines with the hard-margin loss: optimal training via combinatorial {B}enders’ cuts, Journal of Global Optimization 92~(1) (2025) 205--225.

\bibitem{rahmaniani_bd-review_2017}
R.~Rahmaniani, T.~G. Crainic, M.~Gendreau, W.~Rei, The benders decomposition algorithm: A literature review, European Journal of Operational Research 259~(3) (2017) 801--817.

\bibitem{saharidis2011initialization}
G.~K. Saharidis, M.~Boile, S.~Theofanis, Initialization of the benders master problem using valid inequalities applied to fixed-charge network problems, Expert Systems with Applications 38~(6) (2011) 6627--6636.

\bibitem{birge1988multicut}
J.~R. Birge, F.~V. Louveaux, A multicut algorithm for two-stage stochastic linear programs, European Journal of Operational Research 34~(3) (1988) 384--392.

\bibitem{magnanti1981accelerating}
T.~L. Magnanti, R.~T. Wong, Accelerating benders decomposition: Algorithmic enhancement and model selection criteria, Operations research 29~(3) (1981) 464--484.

\bibitem{saharidis2010improving}
G.~K. Saharidis, M.~G. Ierapetritou, Improving benders decomposition using maximum feasible subsystem (mfs) cut generation strategy, Comput. Chem. Eng. 34~(8) (2010) 1237--1245.

\bibitem{fischetti2010note}
M.~Fischetti, D.~Salvagnin, A.~Zanette, A note on the selection of benders’ cuts, Mathematical Programming 124~(1) (2010) 175--182.

\bibitem{sherali2013generating}
H.~D. Sherali, B.~J. Lunday, On generating maximal nondominated benders cuts, Annals of Operations Research 210~(1) (2013) 57--72.

\bibitem{bodur2017strengthened}
M.~Bodur, S.~Dash, O.~G{\"u}nl{\"u}k, J.~Luedtke, Strengthened benders cuts for stochastic integer programs with continuous recourse, INFORMS Journal on Computing 29~(1) (2017) 77--91.

\bibitem{glushko2022shaped}
P.~Glushko, C.~I. F{\'a}bi{\'a}n, A.~Koberstein, An l-shaped method with strengthened lift-and-project cuts, Computational Management Science 19~(4) (2022) 539--565.

\bibitem{holmberg1994using}
K.~Holmberg, On using approximations of the benders master problem, European Journal of Operational Research 77~(1) (1994) 111--125.

\bibitem{rei2009accelerating}
W.~Rei, J.-F. Cordeau, M.~Gendreau, P.~Soriano, Accelerating benders decomposition by local branching, INFORMS Journal on Computing 21~(2) (2009) 333--345.

\bibitem{contreras2011benders}
I.~Contreras, J.-F. Cordeau, G.~Laporte, Benders decomposition for large-scale uncapacitated hub location, Operations research 59~(6) (2011) 1477--1490.

\bibitem{li2013inexact}
M.~Li, L.~N. Vicente, Inexact solution of nlp subproblems in minlp, Journal of Global Optimization 55~(4) (2013) 877--899.

\bibitem{tsang_inexact_2022}
M.~Y. Tsang, K.~S. Shehadeh, F.~E. Curtis, An inexact column-and-constraint generation method to solve two-stage robust optimization problems, Operations Research Letters 51~(1) (2023) 92--98.

\bibitem{boyd2011distributed}
S.~Boyd, N.~Parikh, E.~Chu, B.~Peleato, J.~Eckstein, et~al., Distributed optimization and statistical learning via the alternating direction method of multipliers, Foundations and Trends{\textregistered} in Machine learning 3~(1) (2011) 1--122.

\bibitem{xie2017inexact}
J.~Xie, A.~Liao, X.~Yang, An inexact alternating direction method of multipliers with relative error criteria, Optimization Letters 11~(3) (2017) 583--596.

\bibitem{easwaran2009tabu}
G.~Easwaran, H.~{\"U}ster, Tabu search and benders decomposition approaches for a capacitated closed-loop supply chain network design problem, Transportation science 43~(3) (2009) 301--320.

\bibitem{poojari_genetic_2009}
C.~Poojari, J.~Beasley, Improving benders decomposition using a genetic algorithm, European Journal of Operational Research 199~(1) (2009) 89--97, publisher: Elsevier {BV}.

\bibitem{jiang_tabu-bd_2009}
W.~Jiang, L.~Tang, S.~Xue, A hybrid algorithm of tabu search and benders decomposition for multi-product production distribution network design, in: 2009 IEEE International Conference on Automation and Logistics, IEEE, 2009, pp. 79--84.

\bibitem{lai_gene-bd-capa_2010}
M.-C. Lai, H.-s. Sohn, T.-L.~B. Tseng, C.~Chiang, A hybrid algorithm for capacitated plant location problem, Expert Systems with Applications 37~(12) (2010) 8599--8605.

\bibitem{Lai201233}
H.~Sohn, B.~Tseng, D.~L. Bricker, A hybrid benders/genetic algorithm for vehicle routing and scheduling problem, International Journal of Industrial Engineering 19~(1) (2012) 33--46.

\bibitem{geoffrion_multicommodity}
A.~M. Geoffrion, G.~W. Graves, Multicommodity distribution system design by benders decomposition, Management science 20~(5) (1974) 822--844.

\bibitem{adriaensen2022autodac}
S.~Adriaensen, A.~Biedenkapp, G.~Shala, N.~Awad, T.~Eimer, M.~Lindauer, F.~Hutter, Automated dynamic algorithm configuration, Journal of Artificial Intelligence Research 75 (2022) 1633--1699.

\bibitem{jia_benders_2021}
H.~Jia, S.~Shen, Benders cut classification via support vector machines for solving two-stage stochastic programs, INFORMS Journal on Optimization 3~(3) (2021) 278--297.

\bibitem{allen2023improvements}
R.~C. Allen, F.~Iseri, C.~D. Demirhan, I.~Pappas, E.~N. Pistikopoulos, Improvements for decomposition based methods utilized in the development of multi-scale energy systems, Computers \& Chemical Engineering 170 (2023) 108135.

\bibitem{mitrai_gbd_init_2024}
I.~Mitrai, P.~Daoutidis, Taking the human out of decomposition-based optimization via artificial intelligence, part {II}: Learning to initialize, Computers \& Chemical Engineering 186 (2024) 108686.

\bibitem{mitrai2024machine}
I.~Mitrai, P.~Daoutidis, Machine learning-based initialization of generalized benders decomposition for mixed integer model predictive control, in: 2024 American Control Conference (ACC), IEEE, 2024, pp. 4460--4465.

\bibitem{mana_accelerating_nodate}
K.~Mana, S.~Mak, P.~Zehtabi, M.~Cashmore, D.~Magazzeni, M.~Veloso, Accelerating benders decomposition via reinforcement learning surrogate models, FinPlan 2023 (2023) 12.

\bibitem{mitrai_computationally_2024}
I.~Mitrai, P.~Daoutidis, Computationally efficient solution of mixed integer model predictive control problems via machine learning aided benders decomposition, Journal of Process Control 137 (2024) 103207.

\bibitem{shin2019reinforcement}
J.~Shin, T.~A. Badgwell, K.-H. Liu, J.~H. Lee, Reinforcement learning--overview of recent progress and implications for process control, Computers \& Chemical Engineering 127 (2019) 282--294.

\bibitem{nian2020review}
R.~Nian, J.~Liu, B.~Huang, A review on reinforcement learning: Introduction and applications in industrial process control, Computers \& Chemical Engineering 139 (2020) 106886.

\bibitem{yoo2021reinforcement}
H.~Yoo, B.~Kim, J.~W. Kim, J.~H. Lee, Reinforcement learning based optimal control of batch processes using monte-carlo deep deterministic policy gradient with phase segmentation, Computers \& Chemical Engineering 144 (2021) 107133.

\bibitem{hubbs2020deep}
C.~D. Hubbs, C.~Li, N.~V. Sahinidis, I.~E. Grossmann, J.~M. Wassick, A deep reinforcement learning approach for chemical production scheduling, Computers \& Chemical Engineering 141 (2020) 106982.

\bibitem{shin2016multi}
J.~Shin, J.~H. Lee, Multi-time scale procurement planning considering multiple suppliers and uncertainty in supply and demand, Computers \& Chemical Engineering 91 (2016) 114--126.

\bibitem{wang2025risk}
J.~Wang, C.~L. Swartz, K.~Huang, Risk-averse supply chain management via robust reinforcement learning, Computers \& Chemical Engineering 192 (2025) 108912.

\bibitem{burtea2024constrained}
R.~Burtea, C.~Tsay, Constrained continuous-action reinforcement learning for supply chain inventory management, Computers \& Chemical Engineering 181 (2024) 108518.

\bibitem{kotecha2025leveraging}
N.~Kotecha, A.~del Rio~Chanona, Leveraging graph neural networks and multi-agent reinforcement learning for inventory control in supply chains, Computers \& Chemical Engineering (2025) 109111.

\bibitem{shin2019multi}
J.~Shin, J.~H. Lee, Multi-timescale, multi-period decision-making model development by combining reinforcement learning and mathematical programming, Computers \& Chemical Engineering 121 (2019) 556--573.

\bibitem{gao2024deep}
Q.~Gao, A.~M. Schweidtmann, Deep reinforcement learning for process design: Review and perspective, Current Opinion in Chemical Engineering 44 (2024) 101012.

\bibitem{khan2020searching}
A.~Khan, A.~Lapkin, Searching for optimal process routes: A reinforcement learning approach, Computers \& Chemical Engineering 141 (2020) 107027.

\bibitem{mitrai_review-opt-ml_2024}
I.~Mitrai, P.~Daoutidis, Accelerating process control and optimization via machine learning: a review, Reviews in Chemical Engineering 41~(4) (2025) 401--418.

\bibitem{bengio2021machine}
Y.~Bengio, A.~Lodi, A.~Prouvost, Machine learning for combinatorial optimization: a methodological tour d’horizon, European Journal of Operational Research 290~(2) (2021) 405--421.

\bibitem{tang_reinforcement_2020}
Y.~Tang, S.~Agrawal, Y.~Faenza, Reinforcement learning for integer programming: Learning to cut, in: International conference on machine learning, PMLR, 2020, pp. 9367--9376.

\bibitem{wang_cut-hem_2023}
Z.~Wang, X.~Li, J.~Wang, Y.~Kuang, M.~Yuan, J.~Zeng, Y.~Zhang, F.~Wu, Learning cut selection for mixed-integer linear programming via hierarchical sequence model, arXiv preprint arXiv:2302.00244 (2023).

\bibitem{chi_rl-cg_2023}
C.~Chi, A.~Aboussalah, E.~Khalil, J.~Wang, Z.~Sherkat-Masoumi, A deep reinforcement learning framework for column generation, Advances in Neural Information Processing Systems 35 (2022) 9633--9644.

\bibitem{yuan_rl-cg_2024}
H.~Yuan, L.~Fang, S.~Song, A reinforcement-learning-based multiple-column selection strategy for column generation, in: Proceedings of the AAAI Conference on Artificial Intelligence, Vol.~38, 2024, pp. 8209--8216.

\bibitem{sutton1998reinforcement}
R.~S. Sutton, A.~G. Barto, et~al., Reinforcement learning: An introduction, Vol.~1, MIT press Cambridge, 1998.

\bibitem{watkins1992q}
C.~J. Watkins, P.~Dayan, Q-learning, Machine learning 8 (1992) 279--292.

\bibitem{mnih2015dqn}
V.~Mnih, K.~Kavukcuoglu, D.~Silver, A.~A. Rusu, J.~Veness, M.~G. Bellemare, A.~Graves, M.~Riedmiller, A.~K. Fidjeland, G.~Ostrovski, et~al., Human-level control through deep reinforcement learning, nature 518~(7540) (2015) 529--533.

\bibitem{sutton1995sarsa}
R.~S. Sutton, Generalization in reinforcement learning: Successful examples using sparse coarse coding, Advances in neural information processing systems 8 (1995).

\bibitem{williams1992REINFORCE}
R.~J. Williams, Simple statistical gradient-following algorithms for connectionist reinforcement learning, Machine learning 8 (1992) 229--256.

\bibitem{schulman2015trpo}
J.~Schulman, S.~Levine, P.~Abbeel, M.~Jordan, P.~Moritz, Trust region policy optimization, in: International conference on machine learning, PMLR, 2015, pp. 1889--1897.

\bibitem{ppo}
J.~Schulman, F.~Wolski, P.~Dhariwal, A.~Radford, O.~Klimov, Proximal policy optimization algorithms, arXiv preprint arXiv:1707.06347 (2017).

\bibitem{konda1999actor_critic}
V.~Konda, J.~Tsitsiklis, Actor-critic algorithms, Advances in neural information processing systems 12 (1999).

\bibitem{flores2006simultaneous}
A.~Flores-Tlacuahuac, I.~E. Grossmann, Simultaneous cyclic scheduling and control of a multiproduct cstr, Industrial \& engineering chemistry research 45~(20) (2006) 6698--6712.

\bibitem{pyomo}
M.~L. Bynum, G.~A. Hackebeil, W.~E. Hart, C.~D. Laird, B.~L. Nicholson, J.~D. Siirola, J.-P. Watson, D.~L. Woodruff, Pyomo--optimization modeling in python, 3rd Edition, Vol.~67, Springer Science \& Business Media, 2021.

\bibitem{gurobi}
{Gurobi Optimization, LLC}, {Gurobi Optimizer Reference Manual} (2024).

\bibitem{pyomo_dae}
B.~Nicholson, J.~D. Siirola, J.-P. Watson, V.~M. Zavala, L.~T. Biegler, pyomo.dae: a modeling and automatic discretization framework for optimization with differential and algebraic equations, Mathematical Programming Computation 10~(2) (2018) 187--223.

\bibitem{ipopt}
A.~W{\"a}chter, L.~T. Biegler, On the implementation of an interior-point filter line-search algorithm for large-scale nonlinear programming, Mathematical Programming 106~(1) (2006) 25--57.

\bibitem{stable-baselines3}
A.~Raffin, A.~Hill, A.~Gleave, A.~Kanervisto, M.~Ernestus, N.~Dormann, Stable-baselines3: Reliable reinforcement learning implementations, Journal of Machine Learning Research 22~(268) (2021) 1--8.

\end{thebibliography}

\end{document}